\documentclass{amsart}
\usepackage[a4paper,left=20mm,top=5mm,right=20mm,bottom=5mm,includeheadfoot,headheight=20mm, headsep=5mm, footskip=15mm]{geometry}

\usepackage[utf8]{inputenc}
\usepackage[all]{xy}
\usepackage{amsthm,amsmath,amssymb,amsfonts,amscd}
\usepackage[backend=bibtex]{biblatex}
\usepackage{bbm}
\usepackage{booktabs}
\usepackage{csquotes}
\usepackage[english]{babel}
\usepackage{enumerate}
\usepackage{fancyhdr}
\usepackage[final]{pdfpages}
\usepackage{graphicx}
\usepackage{imakeidx}
\usepackage{inconsolata}
\usepackage{latexsym,keyval,ifthen}
\usepackage[makeroom]{cancel}
\usepackage{mathalfa}
\usepackage{mathdots}
\usepackage{multicol}
\usepackage{multicol}
\usepackage{nameref}
\usepackage{newpxmath}
\usepackage{prerex}
\usepackage{pst-node}
\usepackage{randomwalk}
\usepackage{soul}
\usepackage{tgpagella}
\usepackage{theoremref}
\usepackage{tikz}
\usepackage{tikz-cd}
\usetikzlibrary{arrows,arrows.meta}
\usepackage{wrapfig}
\usepackage{xgalley}
\usepackage{yfonts}
\usepackage{calc}

\newtheorem{defn0}{Definition}[section]
\newtheorem{prop0}[defn0]{Proposition}
\newtheorem{thm0}[defn0]{Theorem}
\newtheorem{lemma0}[defn0]{Lemma}
\newtheorem{claim0}[defn0]{Claim}
\newtheorem{corollary0}[defn0]{Corollary}
\newtheorem{example0}[defn0]{Example}
\newtheorem{remark0}[defn0]{Remark}
\newtheorem{assumption0}[defn0]{Assumption}
\newtheorem{conjecture0}[defn0]{Conjecture}
\newtheorem{notation0}[defn0]{Notation}
\newtheorem{question0}[defn0]{Question}

\newenvironment{definition}{\begin{defn0}\rm}{\end{defn0}}

\newenvironment{theorem}{\begin{thm0}}{\end{thm0}}
\newenvironment{lemma}{\begin{lemma0}}{\end{lemma0}}

\newenvironment{corollary}{\begin{corollary0}}{\end{corollary0}}

\newenvironment{remark}{\begin{remark0}\rm}{\end{remark0}}
\newenvironment{assumption}{\begin{assumption0}\rm}{\end{assumption0}}

\newcommand{\ipa}[1]{\left(#1\right)}
\newcommand{\Gal}{{\mathrm {Gal}}}

\newcommand{\ord}{\mathrm{ord}}

\newcommand{\M}{\mathrm{M}}

\newcommand{\Ind}{{\mathrm{Ind}}}
\newcommand{\PGL}{{\mathrm{PGL}}}

\newcommand{\GL}{{\mathrm{GL}}}

\newcommand{\Z}{{\mathbb Z}}
\newcommand{\A}{{\mathbb A}}
\newcommand{\Q}{{\mathbb Q}}
\newcommand{\C}{{\mathbb C}}
\newcommand{\R}{{\mathbb R}}

\newcommand{\N}{{\mathbb N}}

\newcommand{\cA}{{\mathcal A}}

\newcommand{\cF}{{\mathcal F}}

\newcommand{\cG}{{\mathcal G}}

\newcommand{\cP}{{\mathcal P}}
\newcommand{\cO}{{\mathcal O}}

\newcommand{\Hom}{{\mathrm {Hom}}}

\bibliography{biblio}
\begin{document}

\title{Squares of toric period integrals in higher cohomology}
\author{Santiago Molina}

\newcommand{\Addresses}{{
		\bigskip
		\footnotesize
	
		\medskip
		Santiago Molina; Universitat de Lleida\\Campus Universitari Igualada - UdL
Av. Pla de la Massa, 8\\
08700 Igualada, Spain\par\nopagebreak
		\texttt{santiago.molina@udl.cat}
		
}}

\maketitle

\begin{abstract}
Thanks to the Harder-Eichler-Shimura isomorphism we can realize a quaternionic automorphic representation of a fixed weight in the cohomology space of certain arithmetic groups. For many interesting applications, it is convenient to consider the cap-product of a cohomology class in these spaces with a fundamental class associated to a maximal torus.
In the recent paper \cite{preprintsanti2}, we compute the absolute value of such a cap-product, and we relate it to special values of Rankin-Selberg L-functions. This provides a formula analogous to that of Waldspurger in higher cohomology. In this paper we compute the square of the cap-product instead of its absolute value.
\end{abstract}



\section{Introduction}

Let $F$ be any number field, and let $E/F$ be any quadratic extension of discriminant $D$. Let $B$ be a quaternion algebra over $F$ admitting an embedding $E\hookrightarrow B$. Let $\pi$ be an irreducible automorphic representation of $B^\times$ with trivial central character and let $\Pi$ be its Jacquet-Langlands lift to $\PGL_2/F$. The extension $E/F$ provides a maximal torus $T$ in the algebraic group $G$ associated with $B^\times/F^\times$.

Let $\chi$ be a character of $T$.
For many applications it is interesting to compute the period integral 
\[
\int_{\A_E^\times/E^\times\A_F^\times}\chi(t)f(t)d^\times t=\int_{T(\A_F)/T(F)}\chi(t)f(t)d^\times t,
\]
where $d^\times t=\prod_vd^\times t_v$ is the usual Tamagawa measure of $\A_E^\times/E^\times\A_F^\times$ and $f\in \pi$ is any  automorphic form. Indeed, this is precisely the the purpose of the celebrated \emph{Waldspurger formula}, that computes the absolute value of the period integral in terms critical values of the L-function $L(s,\Pi,\chi)$ associated with $\Pi$ and $\chi$. 

Among the many interesting applications of Waldspurger formula, we highlight the case $F=\Q$, $\Pi$ is associated with an elliptic curve, $E/\Q$ is imaginary and $B$ is  definite. In this setting, one can relate the corresponding period integral with the reduction of certain special points in the elliptic curve called Heegner points. This has been used to relate certain Euler systems with critical values of the classical $L$-function $L(s,\Pi)$, providing as a consequence rank zero instances of the \emph{Birch and Swinnerton-Dyer conjecture}. Indeed, one can think of Waldspurger formula as a rank zero counterpart of the celebrated \emph{Gross and Zagier formula}, which is one of the most important breakthroughs towards the Birch and Swinnerton-Dyer conjecture.

Going back to the general setting, let $\Sigma_B$ be the set of split archimedean places of $B$, namely, the set of archimedean completions $F_\sigma\subseteq\C$ where $B\otimes_{F} F_\sigma=\M_2(F_\sigma)$. If we write $r=\#\Sigma_B$, then the (lower) \emph{Harder-Eicher-Shimura isomorphism} realizes the automorphic representation $\pi$ in the $r$-cohomology space of certain arithmetic groups attached to our quaternion algebra $B$. More explicitly, for any $\underline{k}-2=(k_{\tilde\sigma}-2)_{\tilde\sigma}\in (2\N)^{[F:\Q]}$, let us consider the $\C$-vector spaces
\[
V(\underline{k}-2):=\bigotimes_{{\tilde\sigma}:F\hookrightarrow\C}{\rm Sym}^{k_{\tilde\sigma}-2}(\C^2).
\]
Then $V(\underline{k}-2)$ is endowed with a natural action of $G(F)$. For any open compact subgroup $U\subset G(\A_F^\infty)$ of the finite adelic points of $G$, we consider the group cohomology spaces
\begin{equation*}
H^r(G(F)_+,\cA^\infty(V(\underline{k}-2))^U)=\bigoplus_{g\in G(F)_+\backslash G(\A_F^\infty)/U} H^r(\Gamma_g,V(\underline{k}-2)), 
\end{equation*}
where $\Gamma_g=G(F)_+\cap gUg^{-1}$ and $G(F)_+$ is the subgroup of elements with positive norm at all real places. When $\Pi$ has weight $\underline k$, 
the automorphic representation $\pi^\infty=\pi\mid_{G(\A_F^\infty)}$ can be realized in the spaces
\begin{equation*}
H^r_\ast(G(F)_+,\cA^\infty(V(\underline{k}-2)))^{\lambda}=\bigcup_{U\subset G(\A_F^\infty)}H^r(G(F)_+,\cA^\infty(V(\underline{k}-2))^U)^{\lambda},
\end{equation*}
where the superindex is given by a fixed character $\lambda:G(F)/G(F)_+\rightarrow\pm 1$ and stands for the corresponding $\lambda$-isotypical component. Such a realization is induced by a Hecke-equivariant Harder-Eichler-Shimura morphism
\[
{\rm ES}_\lambda:M_{\underline k}(U)\longrightarrow H^r(G(F)_+,\cA^\infty(V(\underline{k}-2))^U)^{\lambda},
\]
where $M_{\underline k}(U)$ is the space of $U$-invariant modular forms for $G$ of weight ${\underline k}$ (see definitions of \S \ref{AFweightk}). 

For many arithmetic applications, it is more convenient to consider the above realization of $\pi^\infty$ in the cohomology of the corresponding arithmetic groups, instead of thinking of it as a space of automorphic forms. Indeed, since the finite dimensional complex vector spaces $V(\underline{k}-2)$ admit rational models, one can show that $\pi^\infty$ is in fact the extension of scalars of a representation defined over a number field, called the coefficient field of $\pi$. Moreover, such rational representation can be realized in the rational analogues of the given cohomology spaces.

On the other hand, canonical homology classes associated to quadratic extensions $E/F$ can be defined using the group of relative units of $E^\times$. Such fundamental classes lie in a $r$-th-homology group ($r=\#\Sigma_B$) if we make the following hypothesis:
\begin{assumption}\label{assuSigmaSigma}
The set $\Sigma_B$ coincides with the set of archimedean places $\sigma$ where $E$ splits. We will write 
\[
\Sigma_B=\Sigma_T^\R\cup\Sigma_T^\C;\qquad \Sigma_T^\R=\{\sigma\in\Sigma_B;\;T(F_\sigma)=\R^\times\},\qquad\Sigma_T^\C=\{\sigma\in\Sigma_B;\;T(F_\sigma)=\C^\times\},
\]
and $r_\R=\#\Sigma_T^\R$, $r_\C=\#\Sigma_T^\C$ with $r=r_\R+r_\C$.
\end{assumption}
Under the previous assumption, we will be able to construct a fundamental class 
\[
\eta\in H_r(T(F)_+,C^0_c(T(\A_F^\infty),\Z)), 
\]
where  $C^0$ denotes the space of locally constant functions, $C^0_c$ the subspace of those with compact support, and $T(F)_+$ is again the subgroup of elements with positive norm at all real places. 

There are plenty of arithmetic applications where both the fundamental class $\eta$ and a cohomology class $\phi^\lambda\in(\pi^\infty)^U\subseteq H^r(G(F)_+,\cA^\infty(V(\underline{k}-2))^U)$ play a very important role. For instance, in \cite{Darmon2001}, \cite{guitart2017automorphic}, \cite{fornea2021plectic} and \cite{HerMol1}, classes $\phi^\lambda$ and $\eta$ in certain precise situations are used to construct $p$-adic points in elliptic curves over $F$. In the most general setting these points are called \emph{plectic points}, and they are conjectured to be defined over precise abelian extensions of $E$. A different but related application is the construction of $p$-adic L-functions. An anticyclotomic $p$-adic L-function is a $p$-adic avatar of the classical L-function $L(s,\Pi,\chi)$. In \cite{fornea2021plectic}, \cite{HerMol1} and \cite{HM2} anticyclotomic $p$-adic L-functions are constructed using $\phi^\lambda$ and $\eta$. In the spirit of the Birch and Swinnerton-Dyer conjecture, it is shown that certain derivatives of these anticyclotomic $p$-adic L-functions are related with the aforementioned plectic points. The main theorem of this note (theorem \ref{mainTHMintro}) is used in  \cite{HerMol1} and \cite{HM2} to obtain precise \emph{interpolation formulas} that link the anticyclotomic $p$-adic L-functions with the classical L-functions $L(s,\Pi,\chi)$.

Let us consider a locally polynomial character $\chi$ such that $\chi\mid_{T(F_\infty)}(t)=t^{\underline m}=\prod_{\tilde\sigma:F\hookrightarrow\C}t_{\tilde\sigma}^{m_{\tilde\sigma}}$ in a neighbourhood of 1, with $\frac{2-k_{\tilde\sigma}}{2}\leq m_{\tilde\sigma}\leq\frac{k_{\tilde\sigma}-2}{2}$. We will show that $\chi$ can be seen as a $T(F)$-invariant element of $C^0(T(\A_F),\C)\otimes V(\underline{k}-2)$.
Moreover, we have a natural $T(F)$-equivariant morphism (see \S \ref{torperiods})
\begin{equation*}\label{introvarphi}
\varphi:\cA^\infty(V(\underline{k}-2))\times \left(C^0(T(\A_F),\C)\otimes V(\underline{k}-2)\right)\longrightarrow C^0(T(\A_F),\C).    
\end{equation*}
Since the Tamagawa measure on $T(\A_F)$ provides a pairing between $C^0(T(\A_F),\C)$ and $C_c^0(T(\A_F^\infty),\Z)$,  for any $T(F)$-invariant locally polynomial character $\chi\in H^0(T(F),C^0(T(\A_F),\C)\otimes V(\underline{k}-2))$ we can consider the cup product
\[
\cP(\phi^\lambda,\chi):=\varphi(\phi^\lambda\cup\chi)\cap\eta\in \C.
\]
We can think of $\cP(\phi^\lambda,\chi)$ as an analogue of the above period integral $\int_{T(\A_F)/T(F)}\chi(t)f(t)d^\times t$ in higher cohomology. In fact, in theorem \ref{mainTHM1} we relate these two concepts by means of the morphism ${\rm ES}_\lambda$. Thus, it is sensible to think that there may be a formula computing the absolute value of $\cP(\phi^\lambda,\chi)$ in analogy with Waldspurger's formula. This is precisely the main content of \cite{preprintsanti2} (see theorem \ref{THMwaldsHC1}). Nevertheless, for some applications such as the interpolation properties previously discussed, it is more convenient to compute $\cP(\phi^\lambda,\chi)^2$ instead of its absolute value, and this is the purpose of this paper.

Let $N\subset\cO_F$ be the conductor of $\Pi$ and let $c\subset\cO_F$ be the conductor of the character $\chi$. As shown in \S \ref{normnewvec}, there exists a normalized modular form $\Psi\in M_{\underline k}(U_0(N))$ generating $\Pi^\infty$, where $U_0(N)\subseteq\PGL_2(\A_F^\infty)$ is the usual subgroup of upper triangular matrices modulo $N$. Thus, it should be interesting to have an analogous test vector for $\pi^\infty$. 
In \cite{CST} a canonical subspace $V(\pi^\infty,\chi)\subset(\pi^\infty)^U$ is defined, where $U=\hat\cO_N^\times$ and $\cO_N\subset B$ is certain \emph{admissible} order of discriminant $N$ (see \S \ref{classWalds}). The space $V(\pi^\infty,\chi)$ is one dimensional when, for all finite places $v$, the local root numbers $\epsilon(1/2,\pi_v,\chi_v)=\eta_{T,v}(-1)\epsilon(B_v)$, where $\eta_T$ is the quadratic character associated with $E/F$, $\epsilon(B_v)=1$ if $B_v$ is a matrix algebra and $\epsilon(B_v)=-1$ otherwise. In corollary \ref{coroWFHC} and remark \ref{remarklastcase} we provide a recipe to compute $\cP(\phi_0^\lambda,\chi)^2$ for all $\phi_0^\lambda\in V(\pi^\infty,\chi)$, but, in order to be more explicit, in the remainder of the introduction we will restrict ourselves to the following case:
\begin{assumption}\label{assuNc}
Assume that, for all finite places $v$, either $\ord_v(c)\geq \ord_v(N)$ or $\ord_v(c)=0$, with $\ord_v(N)\leq 1$ in case $\ord_v(c)=0$ and $E_v$ non-split.
\end{assumption}
In this situation the admissible order $\cO_N$ is an Eichler order (see lemma \ref{lemmaadmiorder}). Moreover, the space $V(\pi^\infty,\chi)$ is simply $(\pi^\infty)^U$. In \cite[\S 3.2]{preprintsanti2}, we introduce a natural bilinear inner product $\langle\;,\;\rangle:M_{\underline{k}}(U)\times M_{\underline{k}}(U)\rightarrow\C$ well defined for any quaternion algebra $B$. 
If $F$ is totally real and $G=\PGL_2$ then $\langle\Phi,\bar\Phi\rangle=2^{\underline k}{\rm vol}(U)\left(\frac{\pi}{2}\right)^d(\Phi,\Phi)_{U}$, where $(\;,\;)_{U}$ is the usual Petersson inner product on  Hilbert modular forms, and the volume is considered with respect to the usual Tamagawa measure on $G$.
The embedding $E\subset B$ provides a decomposition $B=E\oplus EJ_0$, with  $J_0$ in the normalizer of $E$ in $B$ and satisfying $J_0^2=M\in F^\times$. We will show that there exists $k_0\in T(\A_F)$ such that $k_0^{-1}J_0\in G(F_\infty)_+\times w_{S^D}U$, where $w_{S^D}$ is the Atkin-Lehner involution $w_{S^D}=\prod_{v\in S^D}w_v$ with $S^D=\{v\mid N:\;{\rm ord}_v(c)=0;\;v\nmid D\}$. Write $\chi_\infty=\chi\mid_{T(F_\infty)}$, $\chi_f=\chi\mid_{T(\A_F^\infty)}$, and let $\varepsilon(S^D)$ be the eigenvalue of $w_{S^D}$ acting on $(\pi^\infty)^U$. The following result can be easily deduced from theorem \ref{mainTHM}:
\begin{theorem}\label{mainTHMintro}
Let $\chi:T(\A_F)/T(F)\rightarrow\C^\times$ be a locally polynomial character of conductor $c$ such that  $\chi\mid_{T(F_\infty)}(t_\infty)=t_\infty^{\underline{m}}\chi_0(t_\infty)$, for some locally constant character $\chi_0$ and some $\underline{m}\in\Z^{[F:\Q]}$, with $\frac{2-\underline{k}}{2}\leq \underline{m}\leq\frac{\underline{k}-2}{2}$. Then $\cP(\phi^\lambda,\chi)=0$ unless $\chi_0=\lambda$ and the local root number $\epsilon(1/2,\pi_v,\chi_v)=\chi_v\eta_{T,v}(-1)\epsilon(B_v)$, for all finite places $v$. Moreover, if this is the case, then any decomposable $\phi^\lambda=\bigotimes_{v\nmid\infty}'\phi_{v}^\lambda\in\pi^\infty\subset H_\ast^r(G(F),\cA^\infty(V(\underline{k}-2))(\lambda))$ that
 differ from $\phi_0^\lambda\in (\pi^\infty)^{U}$ in a finite set of places $\mathfrak{S}$ satisfies
    \begin{equation*}\label{formulasquares}
        \cP(\phi^\lambda,\chi)^2=\frac{\varepsilon({S^D})2^{\#S_D}L_{c}(1,\eta_{T})^2h^2
\bar C(\underline k,\underline m)}{\lambda(k_{0})\chi_f^{-1}(k_{0})M^{-\underline m}|c^2 D|^{\frac{1}{2}}}\cdot L^S(1/2,\Pi,\chi)\cdot\frac{\langle \Phi,\Phi\rangle}{\langle \Psi,\Psi\rangle}\cdot\frac{{\rm vol}(U_0(N))}{{\rm vol}(U)}\prod_{v\in \mathfrak{S}}\frac{\beta_v(\phi_v^\lambda,J_0)}{\beta_v(\phi_{0,v}^\lambda,J_0)},
    \end{equation*}
    where $S_D=\{v\mid (N,D),\;\ord_v(c)=0\}$, $S:=\{v\mid (N,c)\}$, $h=\#T(\cO_F)_{+,{\rm tors}}$, the form  $\Phi\in M_{\underline k}(U)$ satisfies ${\rm ES}_\lambda(\Phi)=\phi^\lambda$, $L_{c}(1,\eta_{T})$ is the product of local L-functions at places dividing $c$, $L^S(1/2,\Pi,\chi)$ is the L-function with the local factors at places dividing $S\cup\infty$ removed, and
    \begin{eqnarray*}
        \bar C(\underline k,\underline m)&=&(-1)^{\left(\sum_{\sigma\in\Sigma_B,\tilde\sigma\mid\sigma}\frac{k_{\tilde\sigma}}{2}\right)}4^{r_\R}\cdot(-32\pi)^{r_\C}\left(\frac{1}{\pi}\right)^{d-r-r_\C}\prod_{\sigma\mid\infty}\prod_{\tilde\sigma\mid\sigma}\frac{\Gamma(\frac{k_{\tilde\sigma}}{2}-m_{\tilde\sigma})\Gamma(\frac{k_{\tilde\sigma}}{2}+m_{\tilde\sigma})}{(2\pi)^{k_{\tilde\sigma}}},\\
        \beta_v(\phi_v^\lambda,J_0)&=&\int_{T(F_v)}\chi_v(t_v)\frac{\langle\pi_v(t_v)\phi_{v}^\lambda,\pi_v(J_0)\phi_{v}^\lambda\rangle_v}{\langle\phi_{v}^\lambda,\phi_{v}^\lambda\rangle_v} d^\times t_v,
    \end{eqnarray*}
being $\langle\;,\;\rangle_v$ any $G(F_v)$-invariant bilinear inner product and $d^\times t_v$ any Haar measure of $T(F_v)$.
\end{theorem}

\subsection*{Acknowledgements.}
This project has received funding from the projects PID2021-124613OB-I00 and PID2022-137605NB-I00 from Ministerio de Ciencia e innovaci\'on.

\subsection{Notation}\label{not}

We write $\hat\Z:=\prod_{\ell}\Z_\ell$ and $\hat R:=R\otimes\hat\Z$, for any ring $R$.
Let $F$ be a number field with integer ring $\cO_F$, and let $\A_F$ and $\A_F^\infty=\hat\cO_F\otimes\Q$ be its ring of adeles and finite adeles, respectively. For any place $v$ of $F$, we write $F_v$ for the corresponding completion. If $v$ is non-archimedean, namely $v\mid p$, write $\cO_{F,v}$ its integer ring, ${\rm ord}_v(\cdot)$ its valuation, $\kappa_v$ its residue field, $\varpi_v$ a uniformizer, $d_{F_v}$ its different over $\Q_p$ and $q_v=\#\kappa_v$. Write $\Sigma_F$ for the set of archimedean places of $F$. The fact that $\sigma\in\Sigma_F$ will be usually denoted by $\sigma\mid\infty$. For any embedding $\tilde\sigma:F\hookrightarrow\C$ whose equivalence class is $\sigma\in\Sigma_F$, we will write $\tilde\sigma\mid\sigma$. We define 
\[
\Z^{\Sigma_F}:=\left\{\underline{k}=(k_\sigma)_{\sigma\in\Sigma_F}; k_\sigma=(k_{\tilde\sigma})_{\tilde\sigma\mid\sigma}\in \Z^{[F_\sigma:\R]}\right\}\simeq\Z^{[F:\Q]}.
\]
Given $\underline{k}\in\Z^{\Sigma_F}$ and $x\in F_\infty=\prod_{\sigma\mid\infty}F_\sigma$, we write 
\[
x^{\underline{k}}:=\prod_{\sigma\in\Sigma_F}\prod_{\tilde\sigma\mid\sigma} \tilde\sigma(x)^{k_{\tilde\sigma}}.
\]

Let $B$ be a quaternion algebra over $F$ with maximal order $\cO_B$.
Let $G$ be the algebraic group associated with the group of units of $B$ modulo its center, namely, for any $\cO_F$-algebra $R$
\[
G(R):=(\cO_B\otimes_{\cO_F}R)^\times/R^\times.
\]
Write $\Sigma_B$ for the set of split archimedean places of $B$. 

Let $E/F$ be a quadratic extension and assume that we have a fixed embedding $E\hookrightarrow B$ such that $\cO_{c_0}=E\cap \cO_B$ is an order of conductor $c_0$. We write 
\[
T(R):=(\cO_{c_0}\otimes_{\cO_F}R)^\times/R^\times.
\]
This implies that $T\subset G$ as algebraic groups. 

Throughout the paper we will assume that assumption \ref{assuSigmaSigma} holds. Namely, $\Sigma_B$ is also the set of archimedean places where $E/F$ splits. Thus,
\[
\Sigma_B=\Sigma_{T}^{\R}\cup\Sigma_T^\C,\qquad \Sigma_{T}^{\R}=\{\sigma\mid\infty,\;T(F_\sigma)=\R^\times\},\quad\Sigma_{T}^{\C}=\{\sigma\mid\infty,\;T(F_\sigma)=\C^\times\}.
\]
For any $\sigma\in\Sigma_F$, let $T(F_\sigma)_0\subset T(F_\sigma)$ be the intersection of all connected subgroups $N$ for which the quotient $T(F_\sigma)/N$ is compact. Write also $T(F_\sigma)_+$ for the connected component of $1$ in $T(F_\sigma)$. Depending on the ramification type of $\sigma$, we can visualize $T(F_\sigma)$, $T(F_\sigma)_+$, $T(F_\sigma)_0$ in the following table:
\[\]
\begin{center}
\begin{tabular}{ c| c c c c }
ramification &$T(F_\sigma)$ & $T(F_\sigma)_+$&$T(F_\sigma)_0$&$T(F_\sigma)/T(F_\sigma)_0$ \\ 
 \hline
 $\Sigma_{T}^{\R}$ & $\R^\times$ & $\R_+$& $\R_+$& $\pm 1$ \\  
 $\Sigma_{T}^{\C}$ & $\C^\times$ & $\C^\times$ & $\R_+$ & $S^1$ \\
  $\Sigma_F\setminus\Sigma_{B}$ & $\C^\times/\R^\times$ & $\C^\times/\R^\times$ & $1$ & $\C^\times/\R^\times$ 
\end{tabular}
\end{center}
\[\]

Given $t_\infty\in T(F_\infty)=E^\times_\infty/F^\times_\infty$, write $\tilde t_\infty=(\tilde t_\sigma)_\sigma\in E_\infty^\times$ for a representative. If $E_\sigma/F_\sigma$ splits, write $(\tilde t_\sigma)_1, (\tilde t_\sigma)_2 \in F_\sigma^\times$ for the two components of $\tilde t_\sigma\in F_\sigma^\times\times F_\sigma^\times$. If $E_\sigma/F_\sigma$ does not split, write $(\tilde t_\sigma)_1, (\tilde t_\sigma)_2 \in E_\sigma^\times $ for the image of $\tilde t_\sigma$ under the two $F_\sigma$-isomorphisms. In this last non-split case, we choose an embedding $\tilde\sigma_E:E_\sigma\rightarrow\C$ above each $\tilde\sigma$. For any $\underline{m} \in	\Z^{\Sigma_F}$, we write
\[%
t_\infty^{\underline{m}}:=\prod_{\sigma\mid\infty}\prod_{\tilde\sigma\mid\sigma}\tilde\sigma_E\left(\frac{(\tilde t_\sigma)_1}{(\tilde t_\sigma)_2}\right)^{m_{\tilde\sigma}}\in\C.
\]

\subsection{Haar Measures}\label{haarmeasures}

For any number field $F$ and any place $v$, we choose the Haar measure $dx_v^\times$ for $F_v^\times$:
\[
d^\times x_v=\zeta_v(1)|x_v|_v^{-1}dx_v;\qquad\mbox{where}\quad \left\{\begin{array}{ll}
 dx_v\mbox{ is $[F_v:\R]$ times the usual Lebesgue measure,}    &v\mid\infty;  \\
   dx_v\mbox{ is the Haar measure of $F_v$ such that }{\rm vol}(\cO_{F,v})=|d_{F_v}|_v^{1/2},  &v\nmid\infty, 
\end{array}\right.
\]
with $\zeta_v(s)=(1-q_v^{-s})^{-1}$, if $v\nmid\infty$, $\zeta_v(s)=\pi^{-s/2}\Gamma(s/2)$, if $F_v=\R$, and $\zeta_v(s)=2(2\pi)^{-s}\Gamma(s)$, if $F_v=\C$. One easily checks that, if $v$ is non-archimedean, ${\rm vol}(\cO_{F_v}^\times)=|d_{F_v}|_v^{1/2}$ as well.

The product of such measures provides a Tamagawa measure $d^\times x$ on $\A_F^\times/F^\times$. In fact, such Haar measure satisfies
\[
{\rm Res}_{s=1}\int_{x\in\A_F^\times/F^\times,\;|s|\leq 1}|x|^{s-1}d^\times x={\rm Res}_{s=1}\Lambda_F(s),
\]
where $\Lambda_F(s)=\zeta_F(s)\prod_{v\mid\infty}\zeta_{F_v}(s)$, is the completed Riemann zeta function associated with $F$. This implies that, if we choose $d^\times t$ to be the quotient measure for $T(\A_F)/T(F)=\A_E^\times/\A_F^\times E^\times$, one has that ${\rm vol}(T(\A_F)/T(F))=2L(1,\eta_T)$.

For the group $G$ we just choose the usual Haar measure so that ${\rm vol}(G(\A_F)/G(F))=2$.  

\subsection{Finite dimensional representations}\label{findimreps}

Let $k$ be a positive even integer.
Let $\cP(k)={\rm Sym}^k(\C^2)$ be the space of polynomials in 2 variables homogeneous of degree $k$ with $\PGL_2(\C)$-action:
\[
\left(\left(\begin{array}{cc}a&b\\c& d\end{array}\right)P\right)(X,Y)=(ad-bc)^{-\frac{k}{2}}P(aX+cY,bX+dY),\qquad P\in\cP(k).
\]  
Let us denote $V(k)=\cP(k)^\vee$ with dual $\PGL_2(\C)$-action:
\[
(g\mu)(P)=\mu(g^{-1}P),\qquad \mu\in V(k).
\]  
Notice that $V(k)\simeq \cP(k)$ by means of the isomorphism
\begin{equation}\label{dualVP}
V(k)\longrightarrow\cP(k),\qquad\mu\longmapsto\mu((Xy-Yx)^k).
\end{equation}

\subsubsection{Polynomials and torus}\label{polandtor}

Recall that the fixed embedding $\iota:E\hookrightarrow B$, provides an embedding $B\hookrightarrow{\rm M}_2(E)$ that we will fix throughout the article. Indeed, $\iota$ implies that $B=E\oplus EJ_0$, where $J_0$ normalizes $E$ and $J_0^2\in F^\times$. Hence, we can define the embedding
\begin{equation}\label{embEinB}
    B\hookrightarrow\M_2(E);\qquad e_1+e_2J_0\longmapsto \left(\begin{array}{cc}
    e_1 & J_0^2e_2 \\
    \bar e_2 & \bar e_1
\end{array}\right),
\end{equation}
where $(e\mapsto\bar e)\in \Gal(E/F)$ denotes the non-trivial automorphism. 
For a given $\bar\sigma:F\hookrightarrow\C$, the composition of $B\hookrightarrow {\rm M}_2(E)$ and the fixed extension $\bar\sigma_E:E\hookrightarrow\C$ of $\bar\sigma$, gives rise to an embedding $G(F_\sigma)\hookrightarrow\PGL_2(\C)$, if $\bar\sigma\mid\sigma\mid\infty$. This provides an action of $G(F_\infty)$ on $V(\underline k):=V(\underline k)_\C$ and $\cP(\underline k):=\cP(\underline k)_\C$, for any $\underline k=(k_{\bar\sigma})\in (2\N)^{\Sigma_F}$. 
The subspaces $V(\underline k)_{\bar\Q}$and $\cP(\underline k)_{\bar\Q}$ are invariant under the action of the subgroup $G(F)\subseteq G(F_\infty)$.

The composition $E\stackrel{\iota}{\hookrightarrow}B\hookrightarrow {\rm M}_2(E)$ maps $e$ to $\big(\begin{smallmatrix} e&\\&\bar e\end{smallmatrix}\big)$. This implies that we have a $T(F_\infty)$-equivariant morphism
\begin{equation}\label{PtoCoverC}
    \cP(\underline k)\longrightarrow C(T(F_\infty),\C);\qquad \bigotimes_{\bar\sigma} P_{\bar\sigma}\longmapsto\ipa{(t_\sigma)_{\sigma\mid\infty}\mapsto\prod_{\sigma\mid\infty}\prod_{\bar\sigma\mid\sigma}P_{\bar\sigma}\ipa{1,\bar\sigma_E\ipa{\frac{t_\sigma}{\bar t_\sigma}}}\bar\sigma_E\ipa{\frac{t_\sigma}{\bar t_\sigma}}^{-\frac{k_{\bar\sigma}}{2}}}.
\end{equation}

\section{Fundamental classes}\label{fundclass}

Under assumption \ref{assuSigmaSigma},
 the $\Z$-rank of $T(\cO_F)$ is $r=\#\Sigma_B$ by the Dirichlet Unit Theorem.  Write 
\[
T(\cO_F)_+:=T(\cO_F)\cap T(F_\infty)_+,\qquad T(F_\infty):=\prod_{\sigma\mid\infty}T(F_\sigma)_+.
\] 
Notice that $T(F_{\infty})_0$ is isomorphic to $\R^{r}$ by means of the homomorphism $z\mapsto (\log|\sigma z|)_{\sigma\in\Sigma_T}$. Moreover, under such isomorphism the image of $T(\cO_F)_+$ becomes a $\Z$-lattice $\Lambda$ in $\R^r$. Thus, $T(F_{\infty})_0/\Lambda$ is a $r$-dimensional torus and the \emph{fundamental class} $\xi$ is a generator of $H_{r}(T(F_{\infty})_0/\Lambda,\Z)\simeq \Z$. 
As seen in \cite[\S 3]{HerMol1}, we can identify any $H_{r}(T(F_\infty)_0/\Lambda,\Z)$ with a group cohomology space $H_{r}(\Lambda,\Z)$. 
In particular, we can think of the fundamental class as an element $\xi\in H_r(\Lambda,\Z)$. 

Let us consider the compact subgroup $U:=T(\hat\cO_F)\subset T(\A_F^\infty)$, and write
\[
T(F)_+:=T(F)\cap T(F_\infty)_+,\qquad {\rm Cl}(T)_+:=T(\A_F^\infty)/UT(F)_+.
\]
Notice that the class group $ {\rm Cl}(T)_+$ is finite. 
Hence, we can fix preimages in $\bar t_i\in T(\A_F^{\infty})$ of all the elements $t_i\in {\rm Cl}(F)$ and consider the compact set
\[
\cF:=\bigcup_i\bar t_i U\subset T(\A_F^\infty).
\]
The set of continuous functions $C(\cF,\Z)$ has an action of $T(\cO_F)_+$ (since $\cF$ is $U$-invariant) and the characteristic function $1_\cF$ is $T(\cO_F)_+$-invariant. 
Let us consider the canonical class
\[
\eta=1_\cF\cap\xi\in H_r(T(\cO_F)_+,C(\cF,\Z));\qquad 1_\cF\in H^0(T(\cO_F)_+,C(\cF,\Z)),
\]  
where $\xi\in H_r(T(\cO_F)_+,\Z)$ is, by abuse of notation, the image of $\xi$ through the correstriction morphism.

By \cite[lemma 3.3]{HerMol1}, we have an isomorphism of $T(F)$-modules
$\Ind_{T(\cO_F)_+}^{T(F)}(C(\cF,\Z))\simeq C_c^0(T(\A_F),\Z)$,
where $C_c^0(T(\A_F),\Z)$ is the set of functions in $C^0(T(\A_F),\Z)$ that are compactly supported when restricted to $T(\A_F^\infty)$.
Thus, by Shapiro's Lemma one may regard 
\begin{equation}\label{fundclassdef}
    \eta\in H_r(T(F),C_c^{0}(T(\A_F),\Z)).
\end{equation}

\section{Toric period integrals in higher cohomology}

\subsection{Modular forms of weight $\underline k$}\label{AFweightk}

Let $\underline{k}=(k_\sigma)\in (2\N)^{\Sigma_F}$. 
We want to define the space of modular forms for $G$ of weight $\underline k$ an level $U$, where $U\subset G(\A_F^\infty)$ is an open compact subgroup. For this, let us consider $\sigma\in\Sigma_B$ and denote by $D(k_\sigma)$ the $(\mathfrak{gl}_{2,\R},O(2))$-module of discrete series of weight $k_\sigma$ and trivial central character, when $F_\sigma=\R$, or the $(\mathfrak{gl}_{2,\C},U(2))$-module $\pi(\mu_{k_\sigma},\mu_{k_\sigma}^{-1})$ with $\mu_{k_\sigma}(t)=\tilde\sigma_1(t)^{(k_{\tilde\sigma_1}-1)/2}\tilde\sigma_2(t)^{(1-k_{\tilde\sigma_2})/2}$ and $k_\sigma=(k_{\tilde\sigma_1},k_{\tilde\sigma_2})_{\tilde\sigma_i\mid\sigma}$, when $F_\sigma=\C$. 
If we fix an isomorphism $E_\sigma\simeq F_\sigma^2$, then the identification $B\otimes_FE\simeq\M_2(E)$ induced by \eqref{embEinB} provides an isomorphism 
$G(F_\sigma)\simeq\PGL_2(F_\sigma)$ that sends $T(F_\sigma)$ to the diagonal torus. By means of such identifications and the actions defined in \S \ref{polandtor}, the tensor product
\[
D(\underline{k})=\bigotimes_{\sigma\in\Sigma_B}D(k_\sigma)\otimes\bigotimes_{\sigma\in\Sigma_F\setminus\Sigma_B}V(k_\sigma-2),
\]
is a natural $(\cG_\infty,K_\infty)$-module, where $\cG_\infty$ is the Lie algebra of $G(F_\infty)$ and $K_\infty$ is a maximal compact subgroup. Thus, the space of modular forms for $G$ of weight $\underline k$ an level $U$ can be described as
\begin{equation}\label{defS22}
M_{\underline{k}}(U):=\Hom_{(\cG_\infty,K_\infty)}\ipa{D(\underline{k}),\cA(G)^{U}},    
\end{equation}
where $\cA(G)^{U}$ is the space of $U$-invariant automorphic forms for $G/F$. In \cite[\S 3.2]{preprintsanti2}, a natural bilinear inner product $\langle\;,\;\rangle:M_{\underline{k}}(U)\times M_{\underline{k}}(U)\rightarrow\C$ is introduced. In order to provide its precise definition, we need to introduce a natural morphism
\begin{equation}\label{deltasdef}
    \underline{\delta s}_\lambda:V(\underline{k}-2)(\lambda)\longrightarrow D(\underline{k}),
\end{equation}
for a given character $\lambda:G(F_\infty)/G(F_\infty)_+\rightarrow\pm1$.
For any $\sigma\in\Sigma_B$, write $\cG_\sigma$ for the Lie algebra of $G(F_\sigma)$ and $K_\sigma$ for the corresponding maximal compact subgroup. If $\lambda_\sigma=\lambda\mid_{G(F_\sigma)}$, then there exists a non-trivial extension of $(\cG_\sigma,K_\sigma)$-modules
\[
0\longrightarrow D(k_\sigma)\stackrel{\iota_{\lambda_\sigma}}{\longrightarrow} I(\lambda_\sigma)\stackrel{p_{\lambda_\sigma}}{\longrightarrow} V(k_\sigma-2)(\lambda_\sigma)\longrightarrow 0.
\]
Moreover, $p_{\lambda_\sigma}$ admits a unique $K_\sigma$-equivariant section $s_{\lambda_\sigma}$.
If we consider $\delta_\sigma^T={\rm exp}_\sigma(1)$, the image of 1 through the exponential map ${\rm exp}_\sigma:F_\sigma\stackrel{\simeq}{\rightarrow} {\rm Lie}(F_\sigma^\times)={\rm Lie}T(F_\sigma)\subset \cG_\sigma$, we can define
\begin{equation}\label{defdeltaT}
    \delta s_{\lambda_\sigma}:V(k_\sigma-2)(\lambda_\sigma)\longrightarrow D(k_\sigma);\qquad \delta s_{\lambda_\sigma}(\mu)=\iota_{{\lambda_\sigma}}^{-1}(\delta^T_\sigma(s_{\lambda_\sigma}(\mu))-s_{\lambda_\sigma}(\delta^T_\sigma\mu)).
\end{equation}
The tensor product of such $\delta s_{\lambda_\sigma}$ provides $\underline{\delta s}_\lambda$.
Finally, the pairing $\langle\;,\;\rangle$ is given by
\begin{equation*}\label{eqpeterssonprod}
    \langle\Phi_1,\Phi_2\rangle:=\int_{G(F)\backslash G(\A_F)}\Phi_1\Phi_2\left(\underline{\delta s}_\lambda(\Upsilon)\right)(g,g)d^\times g,
\end{equation*}
where $d^\times g$ is the usual Tamagawa measure with volume ${\rm vol}(G(\A_F)/G(F))=2$ and
\[
\bigotimes_{\sigma\mid\infty}\Upsilon_\sigma=\Upsilon=\left|\begin{array}{cc}
    x_1 & y_1 \\
    x_2 & y_2
\end{array}\right|^{\underline k-2}\in \cP(\underline{k}-2)\otimes \cP(\underline{k}-2)\simeq V(\underline{k}-2)\otimes V(\underline{k}-2).
\]
By \cite[lemma 4.10]{preprintsanti2} the above definition of $\langle\;,\;\rangle$ is independent of $\lambda$.
As explained in \cite[remark 3.2]{preprintsanti2}, if $F$ is totally real and $G=\PGL_2$ then we have a natural identification between $M_{\underline{k}}(U)$ and the space of Hilbert modular forms of weight $\underline{k}$. Under this identification, $\langle\Phi,\bar\Phi\rangle=2^{\underline k}{\rm vol}(U)\left(\frac{\pi}{2}\right)^d(\Phi,\Phi)_{U}$, where $(\;,\;)_{U}$ is the usual Petersson inner product.

\subsection{Cohomology of arithmetic groups}\label{cohoAG}

For any finite dimensional $G(F)$-representation $V$ over $\C$ and any open compact subgroup $U\subset G(\A_F^\infty)$, we define $\cA^{\infty}(V)^U$ to be the space of functions
\[
\phi: G(\A_F^{\infty})/U\longrightarrow V.
\]
Notice that $\cA^{\infty}(V)^U$ has natural action of $G(F)$: 
\[
(\gamma\phi)(g)=\gamma\phi(\gamma^{-1}g);\qquad \phi\in \cA^{\infty}(V)^U,\quad \gamma\in G(F).
\]
We denote by $\cA^{\infty}(V)^U(\lambda)$ the twist of $\cA^{\infty}(V)^U$ by a character $\lambda:G(F)\rightarrow \C^\times$. 

As shown in \cite[\S 3.1]{preprintsanti2}, for any character $\lambda:G(F)/G(F)_+\rightarrow{\pm 1}$ there exists a natural Hecke-equivariant morphism
\[
{\rm ES}_\lambda: M_{\underline{k}}(U)\longrightarrow H^{r}(G(F),\cA^{\infty}(V(\underline{k}-2))^U(\lambda)).
\]
By means of ${\rm ES}_\lambda$ we can realize the automorphic representations generated by elements of $M_{\underline{k}}(U)$ in the cohomology space 
\[
H_\ast^{r}(G(F),\cA^{\infty}(V(\underline{k}-2))(\lambda)) =\varinjlim_U H^{r}(G(F),\cA^{\infty}(V(\underline{k}-2))^U(\lambda)).
\]
Notice that this setting fits with the classical scenario because by Shapiro 
\[
H^r(G(F)_+,\cA^\infty(V(\underline{k}-2))^U)=\bigoplus_{g\in G(F)_+\backslash G(\A_F^\infty)/U} H^r(\Gamma_g,V(\underline{k}-2)), \qquad \Gamma_g=G(F)_+\cap gUg^{-1}.
\]
Moreover, by \cite{ESsanti} the morphism ${\rm ES}_\lambda$ fits with the classical Eichler-Shimura morphism (up to possibly a factor, see \cite[remark 3.1]{preprintsanti2}).

\subsection{Toric periods}\label{torperiods}

Let us fix a character $\lambda:G(F)/G(F)_+=G(F_\infty)/G(F_\infty)_+\rightarrow\pm1$. 
On the one hand, for any $\Phi\in M_{\underline{k}}(U)$ we can consider ${\rm ES}_\lambda(\Phi)\in H^{r}(G(F),\cA^{\infty}(V(\underline{k}-2))^U(\lambda))$. On the other hand, we have a natural $T(F)$-equivariant morphism  
\begin{eqnarray*}
\varphi:C^{0}(T(\A_F),\C)\otimes V(\underline{k}-2)\otimes \cA^{\infty}(V(\underline{k}-2))^U(\lambda)&\longrightarrow& C^{0}(T(\A_F),\C),\\
\varphi((f\otimes\mu)\otimes\Phi)(z,t)&=&f(z,t)\cdot\lambda(z)\cdot\Phi(t)(\mu(Xy-Yx)^{\underline{k}-2})
\end{eqnarray*}
for all $z\in T(F_\infty)$, $t\in T(\A_F^{\infty})$. Moreover, thanks to the morphism \eqref{PtoCoverC}, we can regard certain locally polynomial functions as elements of $C^{0}(T(\A_F),\C)\otimes V(\underline{k}-2)$.
In particular, those characters $\chi:T(\A_F)/T(F)\rightarrow\C^\times$ that, when restricted to a small neighborhood of 1 in $T(F_\infty)$, are of the form 
\[
\chi(t_\infty)=t_\infty^{\underline{m}},\qquad \underline{m}\in\Z^{\Sigma_F},\quad \frac{2-\underline{k}}{2}\leq \underline{m}\leq\frac{\underline{k}-2}{2}.
\]
We call these $\chi$ locally polynomial characters of degree at most $\frac{\underline{k}-2}{2}$.
By means of the morphisms \eqref{dualVP} and \eqref{PtoCoverC}, we think of them as an elements $\chi\in H^0(T(F),C^{0}(T(\A_F),\C)\otimes V(\underline{k}-2))$. Hence, we can consider the following cup product:
\[
\varphi\left(\chi\cup {\rm ES}_\lambda(\Phi)\right)\in H^u(T(F),C^0(T(\A_F),\C)).
\]
Notice that the Haar measure of $T(\A_F)$ provides a natural $T(F)$-invariant pairing
\[
C^0(T(\A_F),\C)\times C_c^0(T(\A_F),\Z)\longrightarrow \C \qquad
\langle f,\phi\rangle=\sum_{z\in T(F_\infty)/T(F_\infty)_+}\int_{S}\int_{T(\A_F^\infty)}f(zs,t_f)\phi(zs,t_f)d^\times t_f d^\times s,
\]
Hence, we can consider the cap product with the fundamental class $\eta$ of \eqref{fundclassdef} 
\begin{equation}\label{defcP}
    \cP({\rm ES}_\lambda(\Phi),\chi):=\varphi\left(\chi\cup {\rm ES}_\lambda(\Phi)\right)\cap\eta\in\C.
\end{equation}
\begin{theorem}\cite[theorem 3.10]{preprintsanti2}\label{mainTHM1}
Let $\chi:T(\A_F)/T(F)\rightarrow\C^\times$ be a locally polynomial character  such that  $\chi\mid_{T(F_\infty)}(t_\infty)=t_\infty^{\underline{m}}\chi_0(t_\infty)$, for some $\underline{m}=(m_\sigma)\in\Z^{\Sigma_F}$, with $\frac{2-\underline{k}}{2}\leq \underline{m}\leq\frac{\underline{k}-2}{2}$ and some locally constant character $\chi_0$.
Then we have that, for all $\Phi\in M_{\underline{k}}(U)$,
\[
\varphi\left(\chi\cup {\rm ES}_\lambda(\Phi)\right)\cap\eta=h\cdot \int_{T(\A_F)/T(F)}\Phi(\underline{\delta s}_\lambda(\mu_{\underline{m}}))(t)\cdot\chi(t)d^\times t,
\]
where $h=\#T(\cO_F)_{+,{\rm tors}}$ and $\mu_{\underline m}=\bigotimes_\sigma\mu_{m_\sigma}\in V(\underline{k}-2)$ is such that
\begin{equation*}\label{defmumglobal}
    \mu_{m_\sigma}\left(\left|\begin{array}{cc}X& Y\\  x&y \end{array}\right|^{k_\sigma-2}\right)=x^{\frac{k_\sigma-2}{2}-m_\sigma}y^{\frac{k_\sigma-2}{2}+m_\sigma}.
\end{equation*}
\end{theorem}

\subsection{Admissible orders and test vectors}\label{classWalds}

Let $\pi$ be an automorphic representation of $G$ whose restriction to $G(F_\infty)$ is isomorphic to $D(\underline k)$. Let $\Pi$ be its Jacquet-Langlands lift to $\PGL_2$ and assume that $\Pi$ has conductor $N\subseteq\cO_F$. Recall that, for a fixed character $\lambda$, we can realize the representation $\pi^\infty=\pi\mid_{G(\A_F^\infty)}$ in the cohomology space  $H_\ast^{r}(G(F),\cA^{\infty}(V(\underline{k}-2))(\lambda))$ by means of the Harder-Eichler-Shimura morphism ${\rm ES}_\lambda$.
In this section we will introduce certain orders $\cO_N\subset B$ of discriminant $N$ and certain one dimensional spaces of test vectors $V(\pi^\infty,\chi)\subseteq H^{r}(G(F),\cA^{\infty}(V(\underline{k}-2)^U)(\lambda))$, where $U=\hat\cO_N^\times$. 

Let us consider $c\subseteq\cO_F$ the conductor of the character $\chi$, namely, the biggest ideal such that $\chi$ is trivial on $\hat\cO_{c}^\times/\hat\cO_{F}^\times$, where $\cO_c\subset E$ is the order of conductor $c$. 
We define $S_1:=\{v\mid N \mbox{ nonsplit in }T;\;\ord_v(c)<\ord_v(N)\}$. Let $c_1:=\prod_{\mathfrak{p}\mid c,\mathfrak{p}\not\in S_1}\mathfrak{p}^{\ord_{\mathfrak{p}}c}$ be the $S_1$-off part of $c$. Then, for any finite place $v$, there exists a $\cO_{F,v}$-order $\cO_{N,v}\subset B_v$ of discriminant $N\cO_{F,v}$ such that $\cO_{N,v}\cap E_v=\cO_{c_1,v}$. Such an order $\cO_{N,v}$ is called \emph{admissible for $(N,\chi_v)$} if at places $v\mid (N,c_1)$, the order $\cO_{N,v}$ is the intersection of two maximal orders $\cO_{B,v}'$ and $\cO_{B,v}''$ of $B_v$ such that 
    \[
    \cO_{B,v}'\cap E_v=\cO_{c,v},\qquad \cO_{B,v}''\cap E_v=\left\{\begin{array}{ll}\cO_{c/N,v},&\mbox{if }\ord_v(c/N)\geq 0\\\cO_{E,v},&\mbox{otherwise.}\end{array}\right.
    \]
    \begin{remark}
    Our concept of admissibility coincides with that of \cite{CST} because their
     condition (2) (see \cite[definition 1.3]{CST}) does not apply in our situation because $\chi_v\mid_{F_v^\times}$ is trivial.
\end{remark}
Let $\cO_N\subset B$ be an admissible $\cO_F$-order for $(N,\chi)$, in the sense that, for any finite place $v$ the order $\cO_{N,v}$ is admissible for $(N,\chi_v)$. It follows that $\cO_N$ is an $\cO_F$-order of discriminant $N$ such that $\cO_N\cap E=\cO_{c_1}$.

\begin{definition}\label{defV1}
Write $U^{S_1}=\prod_{v\not\in S_1}\cO_{N,v}^\times$.
Let $V(\pi^\infty,\chi)\subset \pi^\infty$ be the subspace of elements $\bigotimes'_vf_v\in \pi^{U^{S_1}}$  such that
$f_v$ is a $\chi_v^{-1}$-eigenform under $T(F_v)$, for all places $v\in S_1$.
    %
\end{definition}
If we assume that the local root number $\epsilon(1/2,\pi_v,\chi_v)=\chi_v\eta_{T,v}(-1)\epsilon(B_v)$ for all $v\nmid\infty$, 
then the space $V(\pi^\infty,\chi)$ is actually one-dimensional by \cite[proposition 3.7]{CST}. 

An Eichler order in $B$ is the intersection of two maximal orders. In general, the order $\cO_N$ is not an Eichler order. There may be finitely many places $v$ where $\cO_{N,v}=\cO_N\otimes_{\cO_{F}}\cO_{F,v}$ is not an Eichler order (see remark \ref{remEO}). In case that an admissible order $\cO_N$ for $(N,\chi)$ turns out to be an Eichler order, we will say that $\cO_N$ is an \emph{admissible Eichler order}.

\subsection{Normalized newvectors}\label{normnewvec}

In this section we will assume that $G=\PGL_2$ and $\pi=\Pi$. We will define a normalized element $\Psi\in M_{\underline k}(U_0(N))$ generating $\Pi$, where $U_0(N)$ is the usual compact and open subgroup 
\[
U_0(N)=\left\{\big(\begin{smallmatrix} a&b\\c&d\end{smallmatrix}\big)\in \GL_2(\hat\cO_F),\;c\in N\hat\cO_F\right\}.
\]
Let $\Psi\in M_{\underline k}(U_0(N))$ be the form generating $\Pi$, normalized so that 
\[
\Lambda(s,\Pi)=|d_F|^{s-1/2}\int_{\A_F^\times/F^\times}\Psi(\underline {\delta s}(\mu_{\underline 0}))\left(\begin{array}{cc}a&\\&1\end{array}\right)|a|^{s-1/2}d^\times a,
\] 
where $|\cdot|:\A_F^\times\rightarrow\R_+$ is the standard adelic absolute value, $d_F\subset\cO_F$ is the different of $F$ and $\Lambda(s,\Pi)$ is the (completed) global L-function associated with $\Pi$. 
As pointed out in \cite[\S 3.2]{preprintsanti2}, in case $F$ totally real $\Psi$ corresponds to the normalized Hilbert newform under the natural identification between $M_{\underline k}(U_0(N))$ and the space of Hilbert modular forms.

\subsection{Squares of toric period integrals in higher cohomology}

Given a locally polynomial character $\chi$ of degree at most $\frac{\underline{k}-2}{2}$, we can consider 
\[
\cP(\cdot,\chi):\pi^\infty\subset H_\ast^r(G(F),\cA^\infty(V(\underline{k}-2))(\lambda))\longrightarrow\C,\qquad\cP(\phi_\lambda,\chi):=\varphi\left(\chi\cup\phi^\lambda\right)\cap\eta\in\C,
\]
as in \eqref{defcP}.
The complex number $\cP(\phi^\lambda,\chi)$ can be seen as an analogy in higher cohomology of the toric periods $\int_{T(\A_F)/T(F)}f(t)\cdot\chi(t)d^\times t$ of an automorphic form $f$. In fact, theorem \ref{mainTHM1} relate these two concepts. Given $\pi$ and $\chi$, we choose an admissible order $\cO_N$ for $(N,\chi)$ and write $U=\hat\cO_N^\times$. 
The following result can be seen as a higher cohomological analogy of the classical Waldspurger formula: 
\begin{theorem}[\cite{preprintsanti2}]\label{THMwaldsHC1}
Let $\chi:T(\A_F)/T(F)\rightarrow\C^\times$ be a locally polynomial character  such that  $\chi\mid_{T(F_\infty)}(t_\infty)=t_\infty^{\underline{m}}\chi_0(t_\infty)$, for some $\underline{m}=(m_{\tilde\sigma})\in\Z^{\Sigma_F}$, with $\frac{2-\underline{k}}{2}\leq \underline{m}\leq\frac{\underline{k}-2}{2}$ and some locally constant character $\chi_0$.  
Then we have $\cP(\cdot,\chi)=0$ unless $\chi_{0}=\lambda$ and the root number $\epsilon(1/2,\pi_v,\chi_v)=\chi_v\eta_{T,v}(-1)\epsilon(B_v)$, for all $v\nmid\infty$. Moreover in that case, for any pair $\phi_{01}^\lambda\in V(\pi^\infty,\chi)$ and $\phi_{02}^\lambda\in V(\pi^\infty,\chi^{-1})$,
\[
\cP(\phi_{01}^\lambda,\chi)\cdot \cP(\phi_{02}^\lambda,\chi^{-1})=\frac{2^{\#S_D}L_{c_1}(1,\eta_{T})^2h^2
C(\underline k,\underline m)}{|c_1^2 D|^{\frac{1}{2}}}\cdot L^S(1/2,\Pi,\chi)\cdot\frac{\langle \Phi_{01},\Phi_{02}\rangle}{\langle \Psi,\Psi\rangle}\cdot\frac{{\rm vol}(U_0(N))}{{\rm vol}(U)},
\]
where $S:=\{v\mid (N,Dc);\mbox{ if }v\parallel N\mbox{ then }\ord_v(c/N)\geq 0\}$, $S_D:=\{v\mid (N,D);\;\ord_v(c)<\ord_v(N)\}$, $\Phi_{0i}\in M_{\underline k}(U)$ are such that ${\rm ES}_\lambda(\Phi_{0i})=\phi_{0i}^\lambda$ and
    \[
    C(\underline k,\underline m)=(-1)^{\left(\sum_{\sigma\not\in\Sigma_B}\frac{k_{\sigma}-2}{2}\right)}4^{\#\Sigma_T^\R}\cdot(32\pi)^{\#\Sigma_T^\C}\left(\frac{1}{\pi}\right)^{\#(\Sigma_F\setminus\Sigma_B)}\prod_{\sigma\mid\infty}\prod_{\tilde\sigma\mid\sigma}\frac{\Gamma(\frac{k_{\tilde\sigma}}{2}-m_{\tilde\sigma})\Gamma(\frac{k_{\tilde\sigma}}{2}+m_{\tilde\sigma})}{(-1)^{m_{\tilde\sigma}}.(2\pi)^{k_{\tilde\sigma}}}.
    \]
\end{theorem}

For some applications it is interesting to compute $\cP(\phi^\lambda,\chi)^2$ for some $\phi^\lambda\in\pi^\infty\subset H_\ast^r(G(F),\cA^\infty(V(\underline{k}-2))(\lambda))$. In order to do that, we consider an element $J\in G(F)\setminus T(F)$ in the normalizer of $T(F)$. In fact, we can choose $J\in G(F)_+$ since $G(F)/G(F)_+\simeq T(F)/T(F)_+$ and we can replace $J$ by $Jt$ for any $t\in T(F)$.
It is easy to check that $J$ satisfies $J t= t^{-1} J$, for all $t\in T(F)$, and $J^2=1$. If ${\rm ES}_\lambda(\Phi)=\phi^\lambda$ then we have by theorem \ref{mainTHM1} and the $G(F)$-invariance of $\Phi(\underline{\delta s}(\mu_{-\underline m}))$:
\begin{eqnarray*}\label{fconJ}
\frac{\cP(\pi^\infty(J)\phi^\lambda,\chi^{-1})}{h}
&=&\int_{T(\A_F)/T(F)}\chi(z,t)^{-1}\Phi(\underline{\delta s}_\lambda(\mu_{-\underline{m}}))(z,tJ)d^\times zd^\times t\\
&=&\int_{T(\A_F)/T(F)}\chi(z,t)\Phi(\underline{\delta s}_\lambda(\mu_{-\underline{m}}))(JzJ,Jt)d^\times zd^\times t\\
&=&\int_{T(\A_F)/T(F)}\chi(z,t)\Phi(J_1\underline{\delta s}_\lambda(\mu_{-\underline{m}}))(zJ_2,t)d^\times zd^\times t\\
&=&\chi^{-1}(J_2)\int_{T(\A_F)/T(F)}\chi(\tau)\Phi(J_1\underline{\delta s}_\lambda(\mu_{-\underline{m}}))(\tau)d^\times \tau,
\end{eqnarray*}
where $J=J_2J_1\in G(F_\infty)$ with $J_1=(J_{1,\sigma})_\sigma\in K_\infty$ and $J_2=(J_{2,\sigma})_\sigma\in T(F_\infty)_+$, corresponding under the morphism $G(F_{\sigma})\stackrel{\imath_{\tilde\sigma}}{\hookrightarrow}\PGL_2(\C)$ induced by \eqref{embEinB} to
\begin{eqnarray*}
    \imath_{\tilde\sigma}(J_{1,\sigma})_\sigma=w=\left(\begin{array}{cc}
     &-1  \\
  1   & 
\end{array}\right),\quad \imath_{\tilde\sigma}(J_{2,\sigma})=\left(\begin{array}{cc}
   \tilde \sigma M_J  &  \\
     & 1 
\end{array}\right),&&\sigma\in\Sigma_B,\qquad\tilde\sigma\mid\sigma;\\ 
J=J_{1,\sigma},\qquad \imath_{\tilde\sigma}(J_1)_\sigma=\left(\begin{array}{cc}
     &-\tilde\sigma M_J  \\
  1   & 
\end{array}\right),&&\sigma\not\in\Sigma_B,\qquad\tilde\sigma\mid\sigma.
\end{eqnarray*}
for some fixed $M_J\in F^\times$. Notice that $w\imath(t)=\imath(t^{-1})w$, and this implies that $w\delta_\sigma^T=-\delta_\sigma^T w$ for all $\sigma\in\Sigma_B$. From the description of $\underline{\delta s}_\lambda$ given in \eqref{defdeltaT}, we deduce that $J_1\underline{\delta s}_\lambda(\mu_{-m_\sigma})=-\underline{\delta s}_\lambda(J_1\mu_{-m_\sigma})$. Moreover, it is easy to compute that $\big(\begin{smallmatrix}
    &-M_J\\1&
\end{smallmatrix}\big)\mu_{-m_\sigma}=M_J^{-m_\sigma}(-1)^{\frac{k_\sigma-2}{2}+m_\sigma}\mu_{m_\sigma}$. Hence, we obtain
\[
\cP(\pi^\infty(J)\phi^\lambda,\chi^{-1})=M_J^{-\underline m}(-1)^{\frac{\underline k-2}{2}+\underline m}(-1)^{\#\Sigma_B}h\int_{T(\A_F)/T(F)}\chi(\tau)\Phi(\underline{\delta s}(\mu_{\underline{m}}))(\tau)d^\times \tau=M_J^{-\underline m}(-1)^{\frac{\underline k-2}{2}+\underline m}(-1)^{r}\cP(\phi^\lambda,\chi).
\]
As in theorem \ref{THMwaldsHC1}, fix non-zero elements $\phi_{01}^\lambda\in V(\pi^\infty,\chi)$ and $\phi_{02}^\lambda\in V(\pi^\infty,\chi^{-1})$, respectively. Then there exists a finite set of nonarchimedean places $\mathfrak{S}_0$ such that $\pi_v(J)\phi_{01,v}^\lambda\not\in V(\pi_v,\chi_v^{-1})$. 
\begin{corollary}\label{coroWFHC}
Let $\chi:T(\A_F)/T(F)\rightarrow\C^\times$ be a locally polynomial character  such that  $\chi\mid_{T(F_\infty)}(t_\infty)=t_\infty^{\underline{m}}\lambda(t_\infty)$, for some $\underline{m}\in\Z^{\Sigma_F}$, with $\frac{2-\underline{k}}{2}\leq \underline{m}\leq\frac{\underline{k}-2}{2}$.
If the local root number $\epsilon(1/2,\pi_v,\chi_v)=\chi_v\eta_{T,v}(-1)\epsilon(B_v)$ for all $v\nmid\infty$, then
    \[
    \cP(\phi_{01}^\lambda,\chi)^2=\frac{2^{\#S_D}L_{c_1}(1,\eta_{T})^2h^2
\bar C(\underline k,\underline m)}{M_J^{-\underline m}|c_1^2 D|^{\frac{1}{2}}}\cdot L^S(1/2,\Pi,\chi)\cdot\frac{\langle \Phi_{01},\pi(J)\Phi_{01}\rangle}{\langle \Psi,\Psi\rangle}\cdot\frac{{\rm vol}(U_0(N))}{{\rm vol}(U)}\prod_{v\in \mathfrak{S}_0}\frac{\beta_{v}(\phi_{01,v}^\lambda,\pi(J)\phi_{01,v}^\lambda)}{\beta_{v}(\phi_{01,v}^\lambda,\phi_{02,v}^\lambda)},
    \]
    where $\Phi_{01}\in M_{\underline k}(U)$ is such that ${\rm ES}_\lambda(\Phi_{01})=\phi_{01}^\lambda$, and
    \begin{eqnarray*}
        \bar C(\underline k,\underline m)&=&(-1)^{\left(\sum_{\sigma\in\Sigma_B,\tilde\sigma\mid\sigma}\frac{k_{\tilde\sigma}}{2}\right)}4^{r_\R}\cdot(-32\pi)^{r_\C}\left(\frac{1}{\pi}\right)^{d-r-r_\C}\prod_{\sigma\mid\infty}\prod_{\tilde\sigma\mid\sigma}\frac{\Gamma(\frac{k_{\tilde\sigma}}{2}-m_{\tilde\sigma})\Gamma(\frac{k_{\tilde\sigma}}{2}+m_{\tilde\sigma})}{(2\pi)^{k_{\tilde\sigma}}},\\
        \beta_v(f_{1,v},f_{2,v})&=&\int_{T(F_v)}\chi_v(t_v)\frac{\langle\pi_v(t_v)\phi^\lambda_{1,v},\phi^\lambda_{2,v}\rangle_v}{\langle \phi^\lambda_{1,v},\phi^\lambda_{2,v}\rangle_v} d^\times t_v,
    \end{eqnarray*}
being $\langle\;,\;\rangle_v$ any $G(F_v)$-invariant bilinear inner product.
\end{corollary}
\begin{proof}
This result follows directly from theorem \ref{THMwaldsHC1} plus the fact that both
    \[
    \cP(\cdot,\chi)\cdot \cP(\cdot,\chi^{-1}),\int_{T(F_v)}\chi_v(t_v)\langle\pi_v(t_v)\phi^\lambda_{1,v},\phi^\lambda_{2,v}\rangle_vd^\times t_v\in \Hom_{T(F_v)}(\pi_v\otimes\chi_v,\C)\otimes\Hom_{T(F_v)}(\pi_v\otimes\chi_v^{-1},\C),
    \]
    and $\dim_\C\left(\Hom_{T(F_v)}(\pi_v\otimes\chi_v,\C)\otimes\Hom_{T(F_v)}(\pi_v\otimes\chi_v^{-1},\C)\right)=1$
\end{proof}

Let us analyze the set of primes $\mathfrak{S}_0$: If $v\in S_1$ then it is clear 
that $\pi_v(J)V(\pi_v,\chi_v)=V(\pi_v,\chi_v^{-1})$, and so $v\not\in \mathfrak{S}_0$. Thus, assume that $v\not\in S_1$. In this case $V(\pi_v,\chi_v)=V(\pi_v,\chi_v^{-1})=\pi_v^{U_v}$ where $U_v=\cO_{N,v}^\times$ and $\cO_{N,v}$ is admissible for $(N,\chi_v)$. If $v$ splits in $T$ and $0<\ord_v(c)<\ord_v(N)$ then, by \cite[lemma 3.3]{CST}, there are two $T(F_v)$-conjugacy classes of $(N,\chi_v)$-admissible orders and $J$ interchanges them. This implies that $\pi_v(J)\pi_v^{U_v}\neq \pi_v^{U_v}$ and $v\in \mathfrak{S}_0$. Finally, if $v$ is such that $\ord_v(c)\geq \ord_v(N)$ or $\ord_v(c)=0$ then, again by \cite[lemma 3.3]{CST}, the admissible order $\cO_{N,v}$ is unique up to $T(F_v)$-conjugation. Since $J^{-1}\cO_{N,v}J$ is also admissible, there exists $k_{J,v}\in T(F_v)$ such that $J=k_{J,v}w_{J,v}$ where $w_{J,v}$ is in the normalizer of $\cO_{N,v}$. Thus $v\in \mathfrak{S}_0$ if and only if $k_{J,v}\not\in \cO_{c,v}^\times$.

The above description of $\mathfrak{S}_0$ indicates that, under certain hypothesis that excludes the case $v$ splits in $T$ and $0<\ord_v(c)<\ord_v(N)$, we can explicitly compute the local terms of corollary \ref{coroWFHC}.
\begin{lemma}\label{lemmaadmiorder}
    Assume that for all $v\nmid\infty$:
    \begin{itemize}
    \item [(i)]Either $\ord_v(c)\geq \ord_v(N)$ or $\ord_v(c)=0$. 
    \item [(ii)]If $\ord_v(c)=0$ and $T(F_v)$ is non-split then $\ord_v(N)\leq 1$.
    \item [(iii)]The local root number $\epsilon(1/2,\pi_v,\chi_v)=\chi_v\eta_{T,v}(-1)\epsilon(B_v)$. 
    \end{itemize}
    Then $\cO_N$ is an Eichler order and $J=k_Jw_J$, where $k_J\in T(\A_F^\infty)$ and $w_J\in w_{S^D}\hat\cO_{N}^\times$, being $w_{S^D}=\prod_{v\in S^D}w_v$ the Atkin-Lehner involution with $S^D=\{v\mid N:\;{\rm ord}_v(c)=0;\;v\nmid D\}$.
\end{lemma}
\begin{proof}
    We have seen that the admissible order $\cO_N$ is unique up to $T(\A_F^\infty)$-conjugation, hence $J=k_Jw_J$ where $w_J$ is in the normalizer of $\cO_{N}$. We will divide the proof in several cases:
    \begin{itemize}
        \item \emph{Case $T(F_v)$ splits and $\ord_v(c)\geq\ord_v(N)$}: In this situation $B_v\simeq\M_2(F_v)$. Without loss of generality, we can assume that the embedding $\imath: E_v\simeq F_v^2\hookrightarrow\M_2(F_v)$ is given by  (see \cite[proof of proposition 3.12]{CST})
        \[
        \imath(a,b)=\gamma_c^{-1}\left(\begin{array}{cc}
            a &  \\
             & b
        \end{array}\right)\gamma_c;\qquad \gamma_c=\left(\begin{array}{cc}
            1 & x^{-1} \\
             & 1
        \end{array}\right),\qquad {\rm ord}_v(x)={\rm ord}_v(c).
        \]
        Then the usual Eichler order $R_0(N)_v=\big(\begin{smallmatrix}
            \cO_{F,v}&\cO_{F,v}\\N\cO_{F,v}&\cO_{F,v}
        \end{smallmatrix}\big)$ is admissible for $(N,\chi_v)$. Thus, $\cO_{N,v}=k_0^{-1}R_0(N)_v k_0$, for some $k_0\in T(F_v)$ and $\cO_{N,v}$ is an Eichler order. Moreover, there exists $k_1\in T(F_v)$ such that $J=k_1\big(\begin{smallmatrix}
            -1&\\  x&1
        \end{smallmatrix}\big)$. Hence, if we set $k_{J,v}=k_1k_0^2$ then 
        \[
        J=k_1\left(\begin{array}{cc}
            -1&\\  x&1
        \end{array}\right)= k_1k_0^2k_0^{-1}\left(\begin{array}{cc}
            -1&\\  x&1
        \end{array}\right)k_0\in k_{J,v}\cO_{N,v}^\times.
        \]

        \item \emph{Case $T(F_v)$ non-split and $\ord_v(c)\geq\ord_v(N)$}: By \cite[lemma 3.2(5)]{CST} we have $B_v\simeq\M_2(F_v)$. Write $\cO_{E,v}=\cO_{F.v}[\tau]$, where if $E_v/F_v$ ramifies then $\tau$ is an uniformizer. Without loss of generality, we can assume that the embedding $\imath: E_v\hookrightarrow\M_2(F_v)$ is given by
        \begin{equation}\label{imathnonsplit}
            \imath(a+b\tau)=\gamma_c^{-1}\left(\begin{array}{cc}
            a+b{\rm Tr}(\tau) & b{\rm N}(\tau)  \\
            -b & a
        \end{array}\right)\gamma_c;\qquad \gamma_c=\left(\begin{array}{cc}
            x{\rm N}(\tau) &  \\
             & 1
        \end{array}\right),\qquad {\rm ord}_v(x)={\rm ord}_v(c),
        \end{equation}
        being ${\rm Tr},{\rm N}$ the reduced trace and norm, respectively. Similarly as above, $R_0(N)_v$ is admissible for $(N,\chi_v)$. Thus, $\cO_{N,v}=k_0^{-1}R_0(N)_v k_0$, for some $k_0\in T(F_v)$ and $\cO_{N,v}$ is an Eichler order. Moreover, there exists $k_1\in T(F_v)$ such that $J=k_1\big(\begin{smallmatrix}
            -1&\\  x{\rm Tr}(\tau)&1
        \end{smallmatrix}\big)$. Hence, if we set $k_{J,v}=k_1k_0^2$ then 
        $J\in k_{J,v}\cO_{N,v}^\times$.

        \item \emph{Case $T(F_v)$ splits and $0=\ord_v(c)<\ord_v(N)$}. In this situation $B_v\simeq\M_2(F_v)$ as well, and we can assume that $\imath(a,b)=\big(\begin{smallmatrix}
            a &  \\
             & b
        \end{smallmatrix}\big)$. Thus, the Eichler order $R_0(N)_v$ is admissible for $(N,\chi_v)$ and there exists $k_1\in T(F_v)$ such that $J=k_1\big(\begin{smallmatrix}
             & 1 \\
            x & 
        \end{smallmatrix}\big)$ with ${\rm ord}_v(x)={\rm ord}_v(N)$. This implies that $\cO_{N,v}=k_0^{-1}R_0(N)_v k_0$, for some $k_0\in T(F_v)$, and $J=k_{J,v}w_v$, where $k_{J,v}=k_1k_0^2$ and $w_v=k_0^{-1}\big(\begin{smallmatrix}
             & 1 \\
            x & 
        \end{smallmatrix}\big)k_0$ is the usual Atkin-Lehner involution.

        \item \emph{Case $T(F_v)$ non-split, $\ord_v(c)=0$ and $\ord_v(N)= 1$}. By \cite[lemma 3.4]{CST} the admissible module $\cO_{N,v}$ is unique. If $B_v\not\simeq\M_2(F_v)$ then $\cO_{N,v}=\cO_{B,v}$ is the maximal order (in particular it is an Eichler order). Notice that $\cO_{B,v}=\cO_{E,v}+J_0\cO_{E,v}$, where $J_0$ normalizes $\cO_{E,v}$. Moreover, $J_0\in\cO_{B,v}^\times$, if $E_v/F_v$ ramifies, and $J_0=w_v$ the Atkin-Lehner involution, if $E_v/F_v$ is inert. Since $J=k_{J,v}J_0$, the result follows in this case. Assume now that $B_v\simeq\M_2(F_v)$ and the embedding $\imath:E_v\hookrightarrow\M_2(F_v)$ is the one described in \eqref{imathnonsplit}. By \cite[(3.3)]{Gross88} the admissible order $\cO_{N,v}=\cO_{E,v}+I\M_2(\cO_{F,v})$, where $I\subset\cO_{E,v}$ is a non-zero ideal of $E_v$ with ${\rm length}_{\cO_{F,v}}(\cO_{E,v}/I)=1$. This excludes the inert case. Moreover, in the ramified case $\cO_{N,v}=\cO_{E,v}+\tau\M_2(\cO_{F,v})=R_0(N)_v$ is an Eichler order and $J=k_J\big(\begin{smallmatrix}
            -1&\\  {\rm Tr}(\tau)&1
        \end{smallmatrix}\big)\in k_J\cO_{N,v}^\times$.
    \end{itemize} 
\end{proof}

\begin{remark}\label{remEO}
    If we are not in the situation of the above lemma then there is no unique $T(F_v)$-conjugacy class of admissible Eichler orders. Indeed, if $T(F_v)$ splits and $0<\ord_v(c)<\ord_v(N)$ then there are two $T(F_v)$-conjugacy classes of $(N,\chi_v)$-admissible orders (in this case the admissible order is an Eichler order). Finally, if $v\in S_1$ then by \cite[(3.3)]{Gross88} the admissible order $\cO_{N,v}=\cO_{E,v}+I\M_2(\cO_{F,v})$, where $I\subset\cO_{E,v}$ is a non-zero ideal of $E_v$ with ${\rm length}_{\cO_{F,v}}(\cO_{E,v}/I)=\ord_v(N)$, and it is clear that it is not an Eichler order if $\ord_v(N)> 1$.
\end{remark}
Given an Eichler order $\cO_N$, write ${\rm Norm}(\hat\cO_N^\times)$ for the normalizer of $\hat\cO_N^\times$ in $G(\A_F^\infty)$. Now we can state the main result of the paper:
\begin{theorem}\label{mainTHM}
Let $\chi:T(\A_F)/T(F)\rightarrow\C^\times$ be a locally polynomial character  such that  $\chi\mid_{T(F_\infty)}(t_\infty)=t_\infty^{\underline{m}}\lambda(t_\infty)$, for some $\underline{m}\in\Z^{\Sigma_F}$, with $\frac{2-\underline{k}}{2}\leq \underline{m}\leq\frac{\underline{k}-2}{2}$.
    Assume that for all $v\nmid\infty$ the local root number $\epsilon(1/2,\pi_v,\chi_v)=\chi_v\eta_{T,v}(-1)\epsilon(B_v)$, and either $\ord_v(c)\geq \ord_v(N)$ or $\ord_v(c)=0$, with $\ord_v(N)\leq 1$ in case $\ord_v(c)=0$ and $T(F_v)$ non-split. Let $\cO_N$ be an admissible Eichler order for $(N,\chi)$ and write $U=\hat\cO_N^\times$. Then any element $J\in G(F)_+\setminus T(F)_+$ in the normalizer of $T(F)$ is of the form $J=k_J{\rm Norm}(\hat\cO_N)$, where $k_J\in T(\A_F^\infty)$. Moreover, for any decomposable $\phi^\lambda=\bigotimes_{v\nmid\infty}'\phi_{v}^\lambda\in\pi^\infty\subset H_\ast^r(G(F),\cA^\infty(V(\underline{k}-2))(\lambda))$ that
 differ from $\phi_0^\lambda=\bigotimes_{v\nmid\infty}'\phi_{0,v}^\lambda\in (\pi^\infty)^{U}\subseteq H^r(G(F),\cA^\infty(V(\underline{k}-2)^U)(\lambda))$ in a finite set of places $\mathfrak{S}$ we have
    \begin{equation}\label{formulasquares}
        \cP(\phi^\lambda,\chi)^2=\frac{\varepsilon({S^D})2^{\#S_D}L_{c}(1,\eta_{T})^2h^2
\bar C(\underline k,\underline m)}{\chi^{-1}(k_{J})M_J^{-\underline m}|c^2 D|^{\frac{1}{2}}}\cdot L^S(1/2,\Pi,\chi)\cdot\frac{\langle \Phi,\Phi\rangle}{\langle \Psi,\Psi\rangle}\cdot\frac{{\rm vol}(U_0(N))}{{\rm vol}(U)}\prod_{v\in \mathfrak{S}}\frac{\beta_v(\phi_v^\lambda,J)}{\beta_v(\phi_{0,v}^\lambda,J)},
    \end{equation}
    where $S:=\{v\mid (N,c)\}$, $S_D=\{v\mid (N,D),\;\ord_v(c)=0\}$, $\Phi\in M_{\underline k}(U)$ is such that ${\rm ES}_\lambda(\Phi)=\phi^\lambda$, $\varepsilon(S^D)$ is the eigenvalue of the Atkin-Lehner operator $w_{S^D}$ on $(\pi^{\infty})^U$ and
    \[
    \beta_v(\phi_v^\lambda,J)=\int_{T(F_v)}\chi_v(t_v)\frac{\langle\pi_v(t_v)\phi_{v}^\lambda,\pi_v(J)\phi_{v}^\lambda\rangle_v}{\langle\phi_{v}^\lambda,\phi_{v}^\lambda\rangle_v} d^\times t_v.
    \]
\end{theorem}
\begin{proof}
    By our hypothesis it is clear that $c=c_1$, $S:=\{v\mid (N,c)\}$ and $S_D=\{v\mid (N,D),\;\ord_v(c)=0\}$. By \cite[proposition 3.8]{CST} and \cite[lemma 3.1(3)]{CST} we have $V(\pi_v,\chi_v)=V(\pi_v,\chi_v^{-1})=\pi_v^{U_v}$ for all $v\nmid\infty$, where $U_v=\cO_{N,v}^\times$. 
    By lemma \ref{lemmaadmiorder}, we can write $J\in k_{J}w_{S^D}U$. Thus,
    \begin{eqnarray*}
        \langle \Phi_{0},\pi(J)\Phi_{0}\rangle\prod_{v\nmid\infty}\frac{\beta_{v}(\phi_{0,v}^\lambda,\pi(J)\phi_{0,v}^\lambda)}{\beta_{v}(\phi_{0,v}^\lambda,\phi_{0,v}^\lambda)}&=&\langle \Phi_{0},\Phi_{0}\rangle\prod_{v\nmid\infty}\frac{\int_{T(F_v)}\chi_v(t_v)\langle\pi_v(t_v)\phi_{0,v}^\lambda,\pi_v(J)\phi_{0,v}^\lambda\rangle_v d^\times t_v}{\int_{T(F_v)}\chi_v(t_v)\langle\pi_v(t_v)\phi_{0,v}^\lambda,\phi_{0,v}^\lambda\rangle_v d^\times t_v}\\
        &=&\varepsilon({S^D})\cdot\chi(k_{J})\cdot\langle \Phi_{0},\Phi_{0}\rangle
    \end{eqnarray*}
    Applying corollary \ref{coroWFHC}, we deduce that 
    \[
    \cP(\phi_{0}^\lambda,\chi)^2=\frac{\varepsilon({S^D})2^{\#S_D}L_{c}(1,\eta_{T})^2h^2
\bar C(\underline k,\underline m)}{\chi^{-1}(k_{J})M^{-\underline m}|c^2 D|^{\frac{1}{2}}}\cdot L^S(1/2,\Pi,\chi)\cdot\frac{\langle \Phi_{0},\Phi_{0}\rangle}{\langle \Psi,\Psi\rangle}\cdot\frac{{\rm vol}(U_0(N))}{{\rm vol}(U)}.
    \]
    Hence, the result follows from the fact that $\dim_\C\Hom_{T(F_v)}\left(\pi_v\otimes\chi_v,\C\right)^{\otimes 2}\leq 1$, by Saito-Tunell, and both
    \[
    \cP(\cdot,\chi)\cdot\cP(\cdot,\chi)\qquad\mbox{and}\qquad\int_{T(F_v)}\chi_v(t_v)\langle\pi_v(t_v)\;\cdot,\pi_v(J)\;\cdot\rangle_v d^\times t_v
    \]
    lie in $\Hom_{T(F_v)}\left(\pi_v\otimes\chi_v,\C\right)^{\otimes 2}$.
\end{proof}
\begin{remark}
    Notice that the left hand side of \eqref{formulasquares} is independent of the choice of $J$, hence, the right hand side should be too. Indeed, if we choose $k\cdot J$ instead for $k\in T(F)_+$ the quotient $\frac{\beta_v(\phi_v,k\cdot J)}{\beta_v(\phi_{0,v},k\cdot J)}=\frac{\chi_v(k)\cdot\beta_v(\phi_v,J)}{\chi_v(k)\cdot\beta_v(\phi_{0,v},J)}=\frac{\beta_v(\phi_v,J)}{\beta_v(\phi_{0,v},J)}$. On the other side, $k_{k\cdot J}=kk_J$ and $M_{k\cdot J}=k\bar k^{-1}M_J$. Hence, if we write $\chi^{\infty}=\chi\mid_{T(\A_F^\infty)}$ then
    \[
    \chi^{-1}(k_{k\cdot J})M_{k\cdot J}^{-\underline m}=\chi^\infty(k)^{-1}k^{-\underline m}\chi^{-1}(k_{J})M_{J}^{-\underline m}=\chi(k)^{-1}\chi^{-1}(k_{J})M_{J}^{-\underline m}=\chi^{-1}(k_{J})M_{J}^{-\underline m},
    \]
    because $k\in T(F)_+$.
\end{remark}

\begin{remark}\label{remarklastcase}
    The only local factor not computed in corollary \ref{coroWFHC} is the case $v\not\in S_1$ and $0<\ord_v(c)<\ord_v(N)$. But one can also compute such local factor at the cost of not being so explicit. Indeed, in this situations we can assume that the admissible order $\cO_{N,v}$ is $T(F_v)\simeq F_v^\times$-conjugated to the usual (upper-triangular modulo $N$) Eichler order $R_0(N)_v$ and the embedding $\imath:F_v^2\hookrightarrow\M_2(F_v)$ is either $\imath_1$ or $\imath_2$
    \[
\iota_1(a,b)=\gamma_c^{-1}\left(\begin{array}{cc}
    a &  \\
     & b
\end{array}\right)\gamma_c;\qquad \iota_2(a,b)=\gamma_c^{-1}\left(\begin{array}{cc}
    b &  \\
     & a
\end{array}\right)\gamma_c;\qquad \gamma_c=\left(\begin{array}{cc}
    1 & x^{-1} \\
     & 1
\end{array}\right),\quad{\rm ord}_v(x)=\ord_v(c).
\]
If we consider the realization of $\pi_v$ in its Whittaker model for some additive character $\psi_v$ and  $W_v\in \pi_v^{R_0(N)_v}$, one obtains that, up to a factor of the form $\chi_v(k)$ for some $k\in T(F_v)$, 
\begin{eqnarray*}
    \frac{\langle \phi_{0,v}^\lambda,\pi_v(J)\phi_{0,v}^\lambda\rangle_v}{\langle \phi_{0,v}^\lambda,\phi_{0,v}^\lambda\rangle_v}\frac{\beta_{v}(\phi_{0,v}^\lambda,\pi_v(J)\phi_{0,v}^\lambda)}{\beta_{v}(\phi_{0,v}^\lambda,\phi_{0,v}^\lambda)}&=&\frac{\int_{T(F_v)}\chi_v(t_v)\langle\pi_v(t_v)W_v,\pi_v(J)\overline{W_v}\rangle_v d^\times t_v}{\int_{T(F_v)}\chi_v(t_v)\langle\pi_v(t_v)W_v,\overline{W_v}\rangle_v d^\times t_v}\\    
&=&\frac{\int_{F_v^\times}\chi_v^{\pm 1}(t_v)\langle\pi_v\big(\big(\begin{smallmatrix}t_v&\\&1\end{smallmatrix}\big)\gamma_c\big)W_v,\pi_v\big(w\gamma_c\big)\overline{W_v}\rangle_v d^\times t_v}{\int_{F_v^\times}\chi_v^{\pm 1}(t_v)\langle\pi_v\big(\big(\begin{smallmatrix}t_v&\\&1\end{smallmatrix}\big)\gamma_c\big)W_v,\pi_v(\gamma_c)\overline{W_v}\rangle_v d^\times t_v}\\ &=&\frac{\int_{F_v^\times}\chi_v^{\pm 1}(t_v)\int_{F_v^\times}\left(\pi_v\big(\gamma_c\big)W_v\right)\big(\begin{smallmatrix}at_v&\\&1\end{smallmatrix}\big)\left(\pi_v\big(w\gamma_c\big)\overline{W_v}\right)\big(\begin{smallmatrix}a&\\&1\end{smallmatrix}\big)d^\times a d^\times t_v}{\int_{F_v^\times}\chi_v^{\pm 1}(t_v)\int_{F_v^\times}\left(\pi_v\big(\gamma_c\big)W_v\right)\big(\begin{smallmatrix}at_v&\\&1\end{smallmatrix}\big)\left(\pi_v(\gamma_c)\overline{W_v}\right)\big(\begin{smallmatrix}a&\\&1\end{smallmatrix}\big)d^\times a d^\times t_v}\\ 
    &=&\frac{Z\left(\frac{1}{2},\pi_v\big(w\gamma_c\big)\overline{W_v},\chi_v^{\mp1}\right)}{Z\left(\frac{1}{2}, \pi_v(\gamma_c)\overline{W_v},\chi_v^{\mp1}\right)}=\gamma\left(\frac{1}{2},\pi_v,\chi_v,\bar\psi_v\right)\frac{Z\left(\frac{1}{2},\pi_v\big(\gamma_c\big)\overline{W_v},\chi_v^{\pm 1}\right)}{Z\left(\frac{1}{2}, \pi_v(\gamma_c)W_v^-,\chi_v^{\mp 1}\right)},
\end{eqnarray*}
where $w=\big(\begin{smallmatrix}&-1\\1&\end{smallmatrix}\big)$, $Z(s,W,\chi)=\int_{F_v^\times}\chi(t_v)|t_v|^{s-\frac{1}{2}}W\big(\begin{smallmatrix}t_v&\\&1\end{smallmatrix}\big)d^\times t_v$ and $\gamma\left(s,\pi_v,\chi_v,\bar\psi_v\right)$ is the factor of the local functional equation (see \cite[theorem 4.7.5]{Bump}). Since $\ord_v(c)=\ord_v(x)>0$ and $\overline{W_v}\big(\begin{smallmatrix}a&\\&1\end{smallmatrix}\big)$ is $\cO_{F,v}^\times$-invariant, we have
\begin{eqnarray*}
    Z\left(\frac{1}{2}, \pi_v(\gamma_c)\overline{W_v},\chi_v^{\pm 1}\right)&=&\int_{F_v^\times}\chi_v^{\pm 1}(a)\left(\pi_v\big(\gamma_c\big)\overline{W_v}\right)\big(\begin{smallmatrix}a&\\&1\end{smallmatrix}\big) d^\times a=\int_{F_v^\times}\chi_v^{\pm 1}(a)\bar\psi_v(x^{-1}a)\overline{W_v}\big(\begin{smallmatrix}a&\\&1\end{smallmatrix}\big) d^\times a\\
    &=&\sum_{k\in \Z}\overline{W_v}\big(\begin{smallmatrix}\varpi_v^k&\\&1\end{smallmatrix}\big)\int_{\cO_{F,v}^\times}\chi_v^{\pm 1}(\varpi_v^ka)\bar\psi_v(x^{-1}\varpi_v^{k}a) d^\times a,
\end{eqnarray*}
where $\varpi_v$ is a uniformizer of $F_v$. By \cite[lemma 2.2]{Spiess} we have that $\int_{\cO_{F,v}^\times}\chi_v^{\pm 1}(a)\bar\psi_v(x^{-1}\varpi_v^{k}a) d^\times a=0$ unless $\ker\psi_v=\varpi_v^{k}\cO_{F,v}$. Hence, if $\ker\psi_v=\varpi_v^{k_0}\cO_{F,v}$ then 
\[
 \frac{\langle \phi_{0,v}^\lambda,\pi_v(J)\phi_{0,v}^\lambda\rangle_v}{\langle \phi_{0,v}^\lambda,\phi_{0,v}^\lambda\rangle_v}\frac{\beta_{v}(\phi_{0,v}^\lambda,\pi_v(J)\phi_{0,v}^\lambda)}{\beta_{v}(\phi_{0,v}^\lambda,\phi_{0,v}^\lambda)}=\gamma\left(\frac{1}{2},\pi_v,\chi_v,\bar\psi_v\right)\chi_v^{\pm 1}(\varpi_v^{2k_0})\frac{\tau(\chi_v^{\pm 1},\bar\psi_v,x)}{\tau(\chi_v^{\mp 1},\bar\psi_v,x)},
\]
where $\tau(\chi_v,\bar\psi_v,x)=\int_{\cO_{F,v}^\times}\chi_v(a)\bar\psi_v(x^{-1}\varpi_v^{k_0}a) d^\times a$ is a Gauss sum.
\end{remark}




\Addresses

\printbibliography

\end{document}